\documentclass[12pt]{amsart}

\usepackage{amsxtra,amssymb,amsthm,amsmath,amscd,url,listings}
\usepackage[utf8]{inputenc}
\usepackage{eucal}
\usepackage{fullpage}
\usepackage{scrtime}
\usepackage{supertabular}
\usepackage{mathrsfs}

\usepackage[colorlinks]{hyperref}
\usepackage{graphicx}

\renewcommand{\mathcal}{\mathscr}

\def\setminus{\mathchoice
  {\mathbin{\vrule height .72ex width 1.61ex depth -.38ex}}
  {\mathbin{\vrule height .72ex width 1.61ex depth -.38ex}}
  {\mathbin{\vrule height .50ex width 0.85ex depth -.28ex}}
  {\mathbin{\vrule height .20ex width 0.570ex depth -.24ex}}
}

\DeclareMathSymbol{A}{\mathalpha}{operators}{`A}%
\DeclareMathSymbol{B}{\mathalpha}{operators}{`B}%
\DeclareMathSymbol{C}{\mathalpha}{operators}{`C}%
\DeclareMathSymbol{D}{\mathalpha}{operators}{`D}%
\DeclareMathSymbol{E}{\mathalpha}{operators}{`E}%
\DeclareMathSymbol{F}{\mathalpha}{operators}{`F}%
\DeclareMathSymbol{G}{\mathalpha}{operators}{`G}%
\DeclareMathSymbol{H}{\mathalpha}{operators}{`H}%
\DeclareMathSymbol{I}{\mathalpha}{operators}{`I}%
\DeclareMathSymbol{J}{\mathalpha}{operators}{`J}%
\DeclareMathSymbol{K}{\mathalpha}{operators}{`K}%
\DeclareMathSymbol{L}{\mathalpha}{operators}{`L}%
\DeclareMathSymbol{M}{\mathalpha}{operators}{`M}%
\DeclareMathSymbol{N}{\mathalpha}{operators}{`N}%
\DeclareMathSymbol{O}{\mathalpha}{operators}{`O}%
\DeclareMathSymbol{P}{\mathalpha}{operators}{`P}%
\DeclareMathSymbol{Q}{\mathalpha}{operators}{`Q}%
\DeclareMathSymbol{R}{\mathalpha}{operators}{`R}%
\DeclareMathSymbol{S}{\mathalpha}{operators}{`S}%
\DeclareMathSymbol{T}{\mathalpha}{operators}{`T}%
\DeclareMathSymbol{U}{\mathalpha}{operators}{`U}%
\DeclareMathSymbol{V}{\mathalpha}{operators}{`V}%
\DeclareMathSymbol{W}{\mathalpha}{operators}{`W}%
\DeclareMathSymbol{X}{\mathalpha}{operators}{`X}%
\DeclareMathSymbol{Y}{\mathalpha}{operators}{`Y}%
\DeclareMathSymbol{Z}{\mathalpha}{operators}{`Z}%

\renewcommand{\leq}{\leqslant}
\renewcommand{\geq}{\geqslant}

\renewcommand{\le}{\leqslant}
\renewcommand{\ge}{\geqslant}

\renewcommand{\phi}{\varphi}

\newcommand{\Cc}{\mathbf{C}}

\newcommand{\Aa}{\mathbf{A}}

\newcommand{\Zz}{\mathbf{Z}}

\newcommand{\Gg}{\mathbf{G}}

\newcommand{\Qq}{\mathbf{Q}}

\newcommand{\bFp}{\bar{\mathbf{F}}_p}
\newcommand{\bQl}{\bar{\mathbf{Q}}_{\ell}}

\newcommand{\Ff}{\mathbf{F}}

\newcommand{\mods}[1]{\,(\mathrm{mod}\,{#1})}

\DeclareMathOperator{\hypk}{Kl}

\newcommand{\lra}{\longrightarrow}

\newcommand{\fleche}[1]{\stackrel{#1}{\lra}}

\DeclareMathOperator{\spec}{Spec}

\DeclareMathOperator{\Gal}{Gal}

\DeclareMathOperator{\Tr}{tr}

\DeclareMathOperator{\End}{End}
\DeclareMathOperator{\Aut}{Aut}

\newcommand{\eps}{\varepsilon}
\renewcommand{\rho}{\varrho}

\DeclareMathOperator{\SL}{SL}
\DeclareMathOperator{\GL}{GL}

\DeclareMathOperator{\Sp}{Sp}

\DeclareMathOperator{\SO}{SO}

\DeclareMathOperator{\SU}{SU}
\DeclareMathOperator{\Un}{U}
\DeclareMathOperator{\USp}{USp}

\DeclareMathSymbol{\gena}{\mathord}{letters}{"3C}
\DeclareMathSymbol{\genb}{\mathord}{letters}{"3E}

\theoremstyle{plain}
\newtheorem{theorem}{Theorem}[section]
\newtheorem{lemma}[theorem]{Lemma}
\newtheorem{corollary}[theorem]{Corollary}

\newtheorem{proposition}[theorem]{Proposition}

\theoremstyle{remark}

\theoremstyle{definition}

\newtheorem{definition}[theorem]{Definition}

\newtheorem{remark}[theorem]{Remark}

\newcommand{\mcL}{\mathcal{L}}

\newcommand{\mcF}{\mathcal{F}}
\newcommand{\mctF}{\widetilde{\mathcal{F}}}

\renewcommand{\geq}{\geqslant}
\renewcommand{\leq}{\leqslant}

\begin{document}

\title{Exponential sums, twisted multiplicativity and moments}

\author{E. Kowalski}
\address{ETH Z\"urich -- D-MATH\\
  R\"amistrasse 101\\
  8092 Z\"urich\\
  Switzerland} 
\email{kowalski@math.ethz.ch}

\author{K. Soundararajan}
\address{Department of Mathematics, Stanford University, Stanford, CA 94305}
\email{ksound@stanford.edu}

\date{\today,\ \thistime} 

\begin{abstract}
  We study averages over squarefree moduli of the size of
  exponential sums with polynomial phases. We prove upper bounds on
  various moments of such sums, and obtain evidence of un-correlation of exponential sums associated to 
  different suitably unrelated and generic polynomials. The proofs combine
  analytic arguments with the algebraic interpretation of exponential
  sums and their monodromy groups.
\end{abstract}

\maketitle

\begin{flushright}
  \textit{Dedicated to the memory of Jean Bourgain}
\end{flushright}

\bigskip
\bigskip

\section{Introduction}

Some of Jean Bourgain's many interactions with number theory involved
exponential sums in different ways. Among these, one can mention his
ground-breaking use of ideas from the circle method to solve Bellow's
problems concerning pointwise ergodic theorems at times $f(n)$, where
$f\in \Zz[X]$ is a polynomial (see in particular~\cite{bourgain3,
  bourgain1, bourgain2}) or its combination with bilinear forms in joint
works with A. Kontorovich to study some aspects of the sieve in orbits
beyond a simple appeal to expansion and spectral gaps (see for
instance~\cite{bk}).
We respectfully dedicate this paper to his memory.

\subsection{Exponential sums with polynomials}

This paper is primarily concerned with exponential sums with polynomial
phases.  Let~$f\in\Zz[X]$ be a non-constant polynomial with degree
$d$. For $q\geq 1$ squarefree and~$a$ coprime to~$q$, we define
$$
W(a;q) =
W_f(a;q)=\frac{1}{\sqrt{q}}\sum_{x\mods{q}}e\Bigl(\frac{af(x)}{q}\Bigr),
$$
where the sum is over residue classes modulo~$q$.  For simplicity we restrict attention to square-free $q$, and set 
$W(a;q)=0$ if $q$ is not
square-free or if $(a,q)>1$.   

An application of the Chinese Remainder Theorem shows that the
exponential sums $W(a;q)$ satisfy the following ``twisted
multiplicativity'': if $(q_1, q_2)=1$ then
$$ 
W(a;q_1 q_2)  = W(a\bar{q_1};q_2) W(a\bar{q_2};q_1), 
$$ 
where $q_1 \bar{q_1} \equiv 1 \mod{q_2}$ and
$q_2 \bar{q_2} \equiv 1 \mod{q_1}$.  Apart from finitely many primes,
the Weil bound gives $|W(a;p)| \le (d-1)$, so that
$|W(a;q)| \ll (d-1)^{\omega(q)}$ where $\omega(q)$ denotes the number of
(distinct) prime factors of $q$.  It follows that
$$ 
\sum_{q\le  x} |W(a;q)| \ll \sum_{q\le x} (d-1)^{\omega(q)} \ll x (\log x)^{d-2},    
$$ 
and we seek an improvement over this ``trivial'' bound, as well as
bounds for related mean values such as $\sum_{q\le x} |W(a;q)|^2$.  The
possibility of obtaining such improvements was first recognized by
Hooley, and explored further in the work of Fouvry and Michel~\cite{fm}.
\par

One of our main theorems gives a refinement of these earlier results.
Given a field $K$, we say that a polynomial $f\in K[X]$ is
\emph{decomposable} if there are polynomials $g$ and $h$ in $K[X]$, both
with degree $\ge 2$, such that $f =g\circ h$.  If $f$ cannot be
expressed as such a composition, we call $f$ \emph{indecomposable}.



\begin{theorem}\label{th-main}  Let $f\in \Qq[X]$ be an indecomposable polynomial with~$\deg(f)=d\geq 3$.
  \par
  \emph{(1)} For any~$a\geq 1$, 
  $$ 
  \sum_{q\le x} |W(a;q)|^2 \ll x (\log \log x)^{(d-1)^2}.
  $$
  \par
  \emph{(2)} There
  exists~$\gamma>0$, depending only on~$d$, such that for
  any~$a\geq 1$,
  $$
  \sum_{q\le x} |W(a;q)| \ll \frac{x}{(\log x)^{\gamma}}. 
  $$
  \end{theorem}
  
The implied constants above (and in what follows) are allowed to depend on $f$.  Throughout we ignore linear polynomials 
where $W(a;q)$ is usually $0$, and quadratic polynomials where $|W(a;q)|$ is usually $1$ (since these are quadratic Gauss sums).  
  
The possibility of obtaining non-trivial bounds for
$$
\sum_{q\leq x}|W(a;q)|
$$
(with~$f$ allowed to be a rational function) was first pointed out by
Hooley in~\cite{hooley} in the case of Kloosterman sums.  Introducing ideas from algebraic geometry (notably from the 
work of Katz~\cite{katz-esde}),  Fouvry and Michel~\cite{fm} refined and extended Hooley's work to general exponential sums. 
Under a hypothesis that the polynomial $f$ is generic (in a sense to be made precise below, see also~\cite[H.1, H.2, H.3, H.3']{fm})
Fouvry and Michel proved in~\cite[Th. 1.5]{fm} that
$$
\sum_{q\leq x}|W(a;q)|\ll x(\log \log x)^{k_f-1}
$$
for some explicit integer~$k_f\geq 1$. 
Theorem~\ref{th-main} refines this in two ways: Firstly it applies to a
larger class of polynomials $f$, with the much simpler criterion of
being indecomposable (for instance, if the degree of $f$ is prime, then
$f$ is automatically indecomposable).  Secondly, part (2) gives an
improvement in the exponent of $\log x$ over the corresponding result in
Fouvry and Michel, showing qualitatively that the average of $|W(a;q)|$
over $q\leq x$ tends to~$0$.

The proof of the second part of Theorem~\ref{th-main} relies on the
following result, which may be of independent interest.

\begin{theorem}\label{th-4}
  Let~$f\in\Zz[X]$ of degree~$d\geq 3$. Then, one of the following two
  possibilities holds:
  \par 
  \emph{(1)} The limit
  $$
  \lim_{p\to+\infty} \frac{1}{p}\sum_{a\in \Ff_p^{\times}} |W(a;p)|^4
  \qquad \text{ exists and equals } 2.
  $$
  \par
  \emph{(2)} There exists $\delta >0$ (depending only on~$d$) and a
  subset of primes with positive density $\ge \delta$ on which
  $$ 
   \frac{1}{p}\sum_{a\in \Ff_p^{\times}} |W(a;p)|^4 \ge 3 + O(p^{-1/2}). 
   $$ 
 \end{theorem} 
  
For a generic (again in a sense to be made precise later) polynomial
$f$, the first case of the theorem holds.

\begin{remark}
  The work of Katz~\cite{katz-perversity} contains material from
  which it is likely that one can deduce Theorem~\ref{th-4}. However, in
  view of the different focus and the generality
  of~\cite{katz-perversity}, our independent and slightly more elementary
  proof seems worth including.
\end{remark}

\subsection{Sums of twisted-multiplicative functions}

A key feature of the exponential sums considered above is their twisted
multiplicativity.  In this section we formulate, following
Hooley~\cite{hooley}, Fouvry and Michel~\cite{fm}, and our own recent
paper~\cite{ks}, a general result on bounding averages of twisted
multiplicative functions.


Suppose we are given a function $V$ that associates to
each prime $p$ and each reduced residue class $a \pmod p$ a complex
number $V(a;p)$.  Extend this to a function $V(a;q)$ where $q$ is
square-free and $a\pmod q$ is a reduced residue class by ``twisted
multiplicativity'': that is, if $q=q_1 q_2$ with $(q_1, q_2)=1$ then
\begin{equation} 
  \label{7.1} 
  V(a;q_1 q_2) = V(a \bar{q_1}; q_2) V(a \bar{q_2}; q_1).
\end{equation} 
\par
Set $V(a;q)=0$ if $q$ is not square-free, or if $a$ is not coprime
to~$q$.  For each prime $p$, let $G(p) \ge g(p) \ge 0$ be such that
\begin{equation} 
  \label{7.2} 
  \max_{(a, p)=1} |V(a;p)| \le G(p), \qquad \text{and}
  \qquad \frac 1p \sum_{(a,p)=1} |V(a;p)| \le g(p). 
\end{equation} 
Extend $g$ and $G$ to all square-free integers using multiplicativity,
so that~(\ref{7.2}) remains valid for all~$q$.

The question then is to obtain, under suitable conditions, a bound for
$$
\sum_{q\le x} |V(a;q)|
$$
that improves upon the trivial bound
$$
\sum_{q\le x} |V(a;q)|\leq \sum_{q\le x} G(q).
$$


\begin{theorem}\label{thm5}  Let $M>0$ be such that $G(p)\le M$ for all primes
  $p$.  Then,  for any fixed integer~$a\geq 1$
  and for all large $x$, we have
  $$ 
  \sum_{q\le x} |V(a;q)| \ll \frac{x}{\log x} \prod_{p\le x} \Big( 1+
  \frac{g(p)}{p}\Big) (\log \log x)^{M},
  $$
  where the implied constant may depend on $M$. 
\end{theorem} 

\begin{remark}
  (1) The twisted multiplicativity \eqref{7.1} is naturally connected
  to the Chinese Remainder Theorem via the Fourier transform.  Suppose
  that for each prime $p$ and any residue class $a\pmod p$, we are
  given a complex number $v(a;p)$.  We extend $v$ to square-free
  moduli $q$ and any residue class $a \pmod q$ by means of the Chinese
  Remainder Theorem: that is we set
  $$ 
  v (a; q) = \prod_{p|q} v(a;p). 
  $$  
  Consider now the Fourier transform of $v$:  
  $$ 
  V(a; q) = \sum_{b \mods q} v(b;q) e(ab/q).
  $$ 
  Then $V(a;q)$ satisfies the twisted multiplicative relation
  \eqref{7.1}.

  If $v(a;p)$ corresponds to a probability measure (thus all $v(a;p)$
  are non-negative and $\sum_a v(a;p) =1$) then $|V(a;p)| \le 1$ for all
  $a \pmod p$, so that we may use $G(p)=1$.  Bounding the $L^1$-norm by
  the $L^2$-norm, we may take
  $$ 
  g(p) = \Big( \frac 1p \sum_{a=1}^{p-1} |V(a;p)|^2 \Big)^{\frac 12} =
  \Big( \sum_{a=1}^{p} |v(a;p)|^2 - \frac 1p \Big)^{\frac 12},
  $$ 
  upon using Parseval.
  \par
  (2) In the applications to equidistribution in~\cite{ks}, the
  functions that occur are Weyl sums of the form
  $$
  V(a;q)=\frac{1}{\rho(q)}\sum_{x\in A_q}e\Bigl(\frac{a\cdot x}{q}\Bigr)
  $$
  for some $h\in\Zz^n\setminus \{0\}$, where~$A_q\subset (\Zz/q\Zz)^n$
  are non-empty sets ``defined by the Chinese Remainder Theorem'', and
  $\rho(q)=|A_q|$.
\end{remark}


\subsection{Non-correlation of exponential sums for different polynomials} 

Our next results are attempts to establish that the exponential sums
associated to two different polynomials~$f$ and~$g$ are
uncorrelated. Here we use the notation $W_f(a;q)$ instead of $W(a;q)$ to
keep track of the dependency on the polynomial.  The results here will
depend on polynomials being suitably generic (as in the work of Fouvry
and Michel \cite{fm} mentioned earlier), and we begin by making this
notion precise.

\begin{definition}[Morse polynomial]\label{def-morse}
  Let~$K$ be a field.  A polynomial $f \in K[X]$ of degree~$d\geq 1$ is
  called \emph{Morse} if it has no repeated roots, its derivative $f'$
  is squarefree of degree~$d-1$, and the values of $f$ at the zeros of
  $f'$ (in an algebraic closure of $K$) are distinct.
\end{definition}

\begin{remark}
  The values of $f$ at the zeros of the derivative of $f$ are known as
  \emph{critical values} of~$f$.  Note that when $f'$ is even, the
  critical values appear in pairs $a+f(0)$, $-a+f(0)$ where $a$ is a
  critical value of $f(x)- f(0)$.
  \par
  If $d$ is smaller than the characteristic of~$K$, then the condition
  that $\deg(f')=d-1$ is automatically fullfilled.
  \par
  If~$f$ is a Morse polynomial, then $0$ is not a critical value of~$f$
  (since there would then be a double zero).
\end{remark}

It is easy to check that a Morse polynomial $f$ is indecomposable
over~$K$ (see Lemma~\ref{lm-indec} below).

We recall that in an abelian group $A$, a subset $S\subset A$ is called
\emph{Sidon} if the equation $a+b=c+d$ with $(a,b,c,d)\in S^4$ has
only the obvious solutions where $a\in \{c,d\}$.

We will say that $S\subset A$ is a \emph{symmetric Sidon
  set}
if there exists $\alpha\in A$ such that $S=\alpha-S$, and the equation
$a+b=c+d$ with $(a,b,c,d)\in S^4$ has only the obvious solutions where
$a\in \{c,d\}$ or $b=\alpha-a$.

We require one last item of terminology. For any field~$K$, two
polynomials $f$ and~$g$ in $K[X]$ are \emph{linearly equivalent
  over~$K$} if there exist $a$, $b$, $c$, $d$ in~$K$, with $a$ and $c$
non-zero, such that
$$
g(X)=af(cX+d)+b.
$$
If $a=1$ and $b=0$, then we say that$~f$ and~$g$ are \emph{strictly
  linearly equivalent}.
Note that the sets $V_f$ and $V_g$ of critical values of $f$ and $g$ are
then related by
$$
V_g=aV_f+d.
$$
In particular, if~$V_f$ is a Sidon set (resp. a symmetric Sidon set)
then so is~$V_g$.

\begin{definition}[Sidon--Morse polynomial]\label{def-generic}
  Let~$K$ be a field. A polynomial $f\in K[X]$ of degree~$d\geq 2$ is
  called \emph{Sidon--Morse} if it is Morse and one of the following
  holds:
  \begin{enumerate}
  \item The set of critical values of~$f$ is a Sidon set in the additive
    group of~$K$.
  \item The polynomial~$f$ is linearly equivalent to an odd
    polynomial~$g$ and the set of critical values of~$g$ is a
    symmetric Sidon set in~$K$.
  \end{enumerate}
  \par
  For a polynomial $f\in A[X]$, with $A$ an integral domain, we say that
  $f$ is Morse (or Sidon--Morse) if the definition is satisfied for the
  field of fractions of $A$.
\end{definition}


\begin{remark}
  (1) To distinguish between the two alternatives above, we will say
  that~$f$ is a \emph{symmetric Sidon--Morse polynomial} in the second
  case.
  \par
  (2) It would seem to be more natural to define a symmetric Sidon
  polynomial to be one where the set of critical values of~$f$ is a
  symmetric Sidon set. This condition is implied by our definition, and
  it may in fact be that this is an equivalent definition (at least
  over~$\Qq$), but we do not know if this is the case. We will see how,
  at some crucial point in the proof of Theorem~\ref{th-connect} below,
  this alternative definition is not sufficient to proceed.
  \par
  (3) Any polynomial~$f$ of degree $d\geq 3$ in~$\Zz[X]$ whose
  derivative has Galois group $\mathfrak{S}_{d-1}$ is a (non-symmetric)
  Sidon--Morse polynomial over $\Qq$ (see~\cite[proof of
  Th. 7.10.6]{katz-esde}).  It is then a Sidon--Morse polynomial over
  $\Ff_p$ for all but finitely many~$p$. In particular, a ``generic''
  polynomial in $\Zz[X]$, in a natural sense, is Sidon--Morse over
  $\Qq$.
\end{remark}

\begin{theorem}\label{th-main2}
  \emph{(1)} Let~$f$ and~$g$ be polynomials in~$\Zz[X]$ with degree
  $d_f\geq 3$ and~$d_g$ respectively. Assume that $f$ is Sidon--Morse
  over $\Qq$ and that $d_f>d_g$.
  Then
  $$
  \sum_{q\le x} |W_f(a;q)\overline{W_g(a;q)}|^2\ll x(\log \log x)^A
  $$
  for some~$A$ depending only on $d_f$ and~$d_g$, where the implied
  constant depends on~$f$ and~$g$.
  \par
  \emph{(2)} Let $m\geq 1$ be an integer and let $f_1$, \ldots, $f_m $
  be polynomials of degrees $d_i=\deg(f_i)\geq 3$. Assume that all $f_i$
  are Sidon--Morse polynomials over~$\Qq$ and moreover that for any
  $i\not=j$, the polynomials $f_i$ and~$f_j$ are not linearly equivalent
  over~$\bar{\Qq}$.
  \par
  Let~$s$ be the number of polynomials~$f_i$ such that $f_i$ is a
  symmetric Sidon--Morse polynomial of odd degree~$\geq 5$.  Then
  for~$x\geq 2$, we have
  \begin{gather*}
    \sum_{q\le x} |W_1(a;q)\cdots W_m(a;q)|\ll \frac{x}{(\log x)^{\gamma}}
    \\
    \sum_{q\le x} |W_1(a;q)\cdots W_m(a;q)|^2\ll x(\log \log x)^A
    \\
    \sum_{q\le x} |W_1(a;q)\cdots W_m(a;q)|^4\ll x(\log
    x)^{2^{m-s}3^s-1}(\log \log x)^A
  \end{gather*}
  for some~$\gamma>0$ and some $A\geq 0$ depending only on $m$ and
  $(d_1,\ldots,d_m)$, where $W_i(a;q)=W_{f_i}(a;q)$. The implied
  constants depend on the polynomials.
\end{theorem}

\begin{remark}
  (1) Since the upper-bounds for two polynomials essentially match those
  in Theorem~\ref{th-main}, this result suggests that the exponential
  sums are uncorrelated. However, we cannot prove it rigorously, since
  we would need to prove some matching lower-bound, such as
  $$
  \sum_{q\le x} |W_f(a;q)|^4\gg x(\log \log x)^B
  $$
  for any~$B\geq 1$, for instance. The best current lower-bound that we
  can achieve in general (by adapting the method of Fouvry and
  Michel~\cite[\S 4]{fm}) is
  $$
  \sum_{q\le x} |W_f(a;q)|^4\gg \frac{x}{\log x}(\log \log x)^B
  $$
  for any~$B\geq 1$ (and the best upper-bound that we
  can give for the last sum is
  $$
  \sum_{q\le x} |W_f(a;q)|^4\ll x(\log x)(\log \log x)^A
  $$
  for some $A$).
  \par
  (2) The genericity assumptions that we impose are not the best
  possible. We will investigate related issues in the paper~\cite{ks3},
  where we will describe in particular other classes of polynomials for
  which Theorem~\ref{th-main2} will apply.
  \par
  (3) In another paper, Fouvry and Michel~\cite[Th.\,1.2,\,1.3]{fm2}
  proved that if~$f$ is a Sidon--Morse polynomial, then there are
  infinitely many squarefree integers $q$ with two prime factors such
  that
  $$
  |W_f(a;q)|\leq q^{-\beta}
  $$
  where $\beta>0$ depends only on the degree of~$f$. It would be
  interesting to extend this property to all indecomposable
  polynomials.
\end{remark}








\subsection{Previous work}

Fouvry and Michel also consider rational functions and lower-bounds. In
the case of the Kloosterman sums
$$
\hypk_2(a;q)=\frac{1}{\sqrt{q}}\sum_{(x,q)=1}
e\Bigl(\frac{ax+\bar{x}}{q}\Bigr)
$$
(i.e., $f(x)=x+1/x$), they obtain
$$
\frac{x}{\log x} \exp((\log\log x)^{5/12})\ll \sum_{q\leq
  x}|\hypk_2(a;q)|\ll \frac{x}{(\log x)^{\delta}}
$$
for any $\delta<1-\tfrac{8}{3\pi}$ (see~\cite[Th. 1.2, 1.3]{fm}).
\par
In this particular case, it is known that if we sum the Kloosterman sums
without taking absolute values, one can prove much stronger estimates
using the spectral theory of automorphic forms, like
$$
\sum_{q\leq x}\hypk_2(1;q)\ll x^{2/3+\eps}
$$
for any $\eps>0$ (see, e.g.,~\cite[\S
16.6]{ik}). Patterson~\cite{patterson1} has also proved a strong result
for certain cubic sums, namely for any non-zero integer~$a$, the
asymptotic formula
$$
\sum_{q\leq x} \sum_{0\leq n<q}e\Bigl(\frac{an^3}{q}\Bigr)\sim
c(a)X^{4/3}
$$
holds for some explicit constant $c(a)>0$, and
Patterson~\cite[Conj.\,2.2]{patterson2} has conjectured similar
asymptotic formulas for all cubic polynomials.
\par
It would be of considerable interest to obtain general conditions on a
twisted-multiplicative function $V(a;q)$, bounded at primes, that ensure
a power saving in the sums
$$
\sum_{q\leq x}V(a;q).
$$

\subsection*{Outline of the paper}

We prove Theorem~\ref{thm5} in the next
section. Section~\ref{sec-prelim} gathers a number of properties of
exponential sums with polynomials, and Section~\ref{sec-main} uses these
results to prove Theorem~\ref{th-main}, assuming Theorem~\ref{th-4}. The
latter is proved in Section~\ref{sec-katz}, and
Section~\ref{sec-generic} discusses generic polynomials. In both of
these, we rely heavily on the foundational studies of
Katz. Section~\ref{sec-mult} concludes with the proof of
Theorem~\ref{th-main2}, and Section~\ref{sec-remark} contains some
hopefully enlightening comments concerning parts of the results of Katz
that we use.

\subsection*{Acknowledgments}

E.K. was partially supported by a DFG-SNF lead agency program grant
(grant number 200020L\_175755).  K.S. is partially supported through a
grant from the National Science Foundation, and a Simons Investigator
Grant from the Simons Foundation.  This work was started when K.S. was
a senior Fellow at the ETH Institute for Theoretical Studies, whom he
thanks for their warm and generous hospitality.
\par
We thank W. Sawin for his comments concerning Section~\ref{sec-katz}.

\section{Sums of twisted-multiplicative functions}

Since the proof of Theorem~\ref{thm5} follows the broad plan of our
earlier work (and is not far from that of Fouvry and Michel~\cite[\S
3]{fm}), we shall be brief.

Put $z=x^{1/(\alpha \log\log x)}$ with $\alpha=3(M^2+1)$. We
factor any integer $q\leq x$ as $q=rs$ where all prime factors of~$s$
are~$\leq z$, and all prime factors of~$r$ are~$>z$. We then have
$$
V(a;q)=V(a;rs)=V(\bar{r}a;s)V(\bar{s}a;r)
$$
by twisted multiplicativity, hence
$$
|V(a;q)|\leq G(r)|V(\bar{r}a;s)|.
$$
\par
We handle first the terms where~$s\leq x^{1/3}$. We split the sum
over~$q\leq x$ according to the residue class of~$r$ modulo~$s$,
getting
\begin{align*}
  \sum_{\substack{q\leq x\\s\leq x^{1/3}}}|V(a;q)|
  &\leq
    \sum_{s\leq x^{1/3}}\sum_{r\leq x/s}G(r)|V(\bar{r}a;s)|
  \leq \sum_{s\leq x^{1/3}}\sum_{t\mods{s}}|V(\bar{t}a;s)|
    \sum_{\substack{r\leq x/s\\r\equiv t\mods{s}}}G(r). 
    \end{align*}
  By Shiu's work on the Brun--Titchmarsh Theorem for multiplicative
functions (see ~\cite[Th.\,1]{shiu})  we may bound the sum over $r$ above by 
$$ 
\ll \frac{x/s}{\log (x/s)} \exp \Big( \sum_{z < p \le x} \frac{G(p)}{p} \Big) \ll \frac{x}{s\log x} \Big(\frac{\log x}{\log z}\Big)^M 
\ll \frac{x}{s\log x} (\log \log x)^{M}.
$$  
Therefore 
\begin{align*}
  \sum_{\substack{q\leq x\\s\leq x^{1/3}}}|V(a;q)| &\ll \frac{x}{\log x} (\log \log x)^M \sum_{s\le x^{1/3}} \frac{1}{s\varphi(s)}\sum_{t\mods{s}}|V(\bar{t}a;s)| \\
  &\ll \frac{x}{\log x} (\log \log x)^M \sum_{s\le x^{1/3}} \frac{g(s)}{\varphi(s)} \ll \frac{x}{\log x} (\log \log x)^M\prod_{p\le x} \Big(1 +\frac{g(p)}{p}\Big). 
 \end{align*}
  

We now consider the contribution of
the terms with~$s>x^{1/3}$.  Since $G(p)\leq M$ for all~$p$, 
$$
\sum_{\substack{q\leq x\\s > x^{1/3}}}|V(a;q)| \leq \sum_{r\leq
  x^{2/3}}M^{\omega(r)}\sum_{x^{1/3}<s\leq x/r}M^{\omega(s)}.
$$
Applying the Cauchy--Schwarz inequality and~\cite[Lemma 3.2]{ks} to the
inner sum, we find that
\begin{align*}
\sum_{x^{1/3}<s\leq x/r}M^{\omega(s)} &\ll \Bigl(\sum_{s\leq
  x/r}M^{2\omega(s)}\Bigr)^{1/2} \Bigl(\sum_{x^{1/3}<s\leq
  x/r}1\Bigr)^{1/2} \\
  &\ll \frac{x}{r}(\log x)^{(M^2-1)/2}
\exp\Bigl(-\frac{\log( x/r)}{2\log z}\Bigr) \ll 
\frac{x}{r} (\log x)^{(M^2-1)/2-\alpha/6} \ll \frac{x}{r\log x}. 
\end{align*}
Therefore 
$$
\sum_{\substack{q\leq x\\s> x^{1/3}}}|V(a;q)| \ll \frac{x}{\log x}
\sum_{r\leq x^{2/3}}\frac{M^{\omega(r)}}{r}\ll \frac{x}{\log x} \exp\Big( \sum_{z\le p\le x} \frac{M}{p}\Big) \ll \frac{x}{\log x} (\log \log x)^M.
$$
 
 The proof of Theorem~\ref{thm5} is now complete.
 
\section{Exponential sums of polynomials, preliminary results}
\label{sec-prelim}

In this section we collect together some results on the exponential sums $W_f(a;p)$.  We shall use and expand on some of these 
results in later sections.  First we recall the Weil bound:  if $f \in {\Zz}[X]$ has degree $d \ge 1$ and $(a,p)=1$ then 
\begin{equation} 
\label{3.1} 
|W_f(a;p)| \le (d-1). 
\end{equation} 

Next we quote a result from Shao~\cite[Th. 2.1]{shao}. 

\begin{lemma} \label{lem3.1} Let $f\in \Zz[X]$ be a polynomial of degree $d$.  Let $\kappa$ denote the number of irreducible factors of $f(X)-f(Y) \in {\Qq}[X,Y]$.  Then $\kappa \le \tau(d)$ (the number of divisors of $d$) and for large $x$ we have 
$$ 
\sum_{p\le x} \frac{1}{p} \Big( \frac 1p \sum_{(a,p)=1} |W(a;p)|^2 \Big) = (\kappa -1) \log \log x + O(1). 
$$ 
\end{lemma}
\begin{proof}  The asymptotic for the sum over primes is given in Theorem 2.1 of Shao~\cite{shao}, and the bound on $\kappa$ is described in the remark after Theorem 2.1 there. 
\end{proof}

While Lemma \ref{lem3.1} involves the factorization of $F(X,Y)= (f(X)-f(Y))/(X-Y)$ in ${\Qq}[X,Y]$, it is of greater significance to understand the factorization of $F(X,Y)$ over $\bar{\Qq}[X,Y]$ (or equivalently over $\Cc[X,Y]$).  

\begin{lemma} \label{lem3.2} Let $f\in \Zz[X]$ be a polynomial of
  degree $d$, and suppose that the polynomial
  $F(X,Y) = (f(X)-f(Y))/(X-Y)$ factors into $m$ irreducible factors
  over $\bar{{\Qq}}[X,Y]$.  If $m=1$ then for all $p$ we have
$$ 
\frac 1p \sum_{a\in \Ff_p^{\times}} |W(a;p)|^2 = 1+ O(p^{-1/2}). 
$$
If $m>1$, then there is  a set of primes ${\mathcal P}$ of density $\ge \delta >0$ (with $\delta$ depending only on the degree $d$) such that for $p\in {\mathcal P}$ 
$$ 
\frac 1p \sum_{a\in \Ff_p^\times} |W(a;p)|^2 = m + O(p^{-1/2}). 
$$ 
\end{lemma}

\begin{proof}
  If $m=1$ then the affine curve with equation $F(X,Y)=0$ is
  geometrically irreducible over ${\Qq}$, so that for all large $p$ it
  is geometrically irreducible over $\Ff_p$.  Orthogonality of
  characters and the Riemann Hypothesis for curves over finite fields
  then show that
  $$ 
  \frac 1p \sum_{a\in \Ff_p^{\times}} |W(a;p)|^2 = \frac 1p \Big|\Big\{
  (x,y)\in \Ff_p^2: F(x,y) = 0 \Big\} \Big| = 1 + O(p^{-1/2}).
  $$  
  
  Now suppose $m>1$, and let $K$ be a finite Galois extension of $\Qq$
  such that $F(X,Y)$ factors in $K[X,Y]$ into $m$ different factors,
  each of which is irreducible in $\bar{\Qq}[X,Y]$.  Thus the
  affine curve defined by $F(X,Y)$ is the union of $m$ geometrically
  irreducible curves over $K$.  Note that the degree of the field $K$
  may be bounded in terms of $d$.  We take ${\mathcal P}$ to be the set
  of primes splitting completely in $K$.  By the Chebotarev density
  theorem ${\mathcal P}$ has density $1/[K:\Qq]$, which is bounded away
  from $0$ by an amount depending only on $d$.  For $p\in {\mathcal P}$,
  the $m$ geometrically irreducible components of the curve $F(X,Y)=0$
  are defined over $\Ff_p$, and the Riemann Hypothesis gives here
$$ 
\frac 1p \sum_{a\in \Ff_p^\times} |W(a;p)|^2 = m + O(p^{-1/2}). 
$$ 
\end{proof} 

Our next result is due to Fried~\cite[Th. 1]{fried} (see also the more
elementary account by Turnwald in~\cite[Th. 1]{turn}).  It describes
when the polynomial $F(X,Y) = (f(X)-f(Y))/(X-Y)$ is absolutely
irreducible, i.e., when $m=1$ in the notation of the previous lemma, and
therefore $\kappa=2$ in the notation of Lemma \ref{lem3.1}.

We recall that for any integer~$d\geq 0$, the Dickson
polynomial~$D_d\in \Zz[X,a]$ is defined to be the unique polynomial
such that
$$
D_d(X+aX^{-1},a)=X^d+(a/X)^{d}
$$
(see, e.g.,~\cite[\S 1]{turn}); in particular,
$D_d(X,0)=X^d$.

\begin{proposition}[Fried]\label{pr-fried}
  Let~$f\in\Zz[X]$ with degree~$d\geq 1$ and let
  $$
  F=(f(X)-f(Y))/(X-Y)\in\Qq[X,Y].
  $$
  \par
  \emph{(1)} If~$\deg(f)$ is not an odd prime, then~$F$ is absolutely
  irreducible if and only if~$f$ is indecomposable in~$\Qq[X]$.
  \par
  \emph{(2)} If~$d$ is an odd prime~$\geq 5$, then~$F$ is absolutely
  irreducible if it is not
  linearly equivalent in~$\Qq[X]$ to a Dickson polynomial $D_d(X,a)$.
  \par
  \emph{(3)} If~$d=3$, then~$F$ is absolutely irreducible if and only
  if~$f$ 
  is not linearly equivalent in~$\Qq[X]$ to a Dickson polynomial $D_3(X,0)$.
\end{proposition}

Putting Lemmas \ref{lem3.1}, \ref{lem3.2} and Proposition \ref{pr-fried} together, we arrive at the following corollary. 

\begin{corollary}\label{cor3.4}  Let $f\in \Zz[X]$ be a polynomial of
  degree $d\geq 1$.  If $f$ is indecomposable then for large $x$ we have
  $$ 
  \sum_{p\le x} \frac 1p \Big( \frac 1p \sum_{a\in \Ff_p^\times} |W(a;p)|^2 \Big) = \log \log x+ O(1), 
  $$ 
  whereas if $f$ is decomposable then for large $x$ we have 
  $$ 
  \sum_{p\le x} \frac 1p \Big( \frac 1p \sum_{a\in \Ff_p^\times} |W(a;p)|^2 \Big) \ge 2 \log \log x+ O(1).
  $$ 
\end{corollary}

\begin{proof}
  If $d$ is prime, then $\kappa$ must be $2= \tau(d)$ in Lemma
  \ref{lem3.1}.  Moreover, $f$ is automatically indecomposable, and so
  the stated result holds in this case.  If $f= g\circ h$ is
  decomposable, then $f(X)-f(Y)$ has $(X-Y)$, $(h(X)-h(Y))/(X-Y)$ and
  $(g(h(X))-g(h(Y)))/(h(X)-h(Y))$ as factors, so that $\kappa \ge 3$
  in Lemma \ref{lem3.1} and the stated result holds.  Finally if the
  degree $d$ is composite and $f$ is indecomposable, then the first
  part of Proposition \ref{pr-fried} shows that $(f(X)-f(Y))/(X-Y)$ is
  irreducible in $\bar{\Qq}[X,Y]$ and therefore in $\Qq[X,Y]$. Either
  Lemma \ref{lem3.1} or Lemma \ref{lem3.2} now gives the stated
  result.
\end{proof}

Lastly we consider the behavior of $W(a;p)$ when $f$ is assumed to be Sidon--Morse over ${\Qq}$.  Here the work of Katz permits a very precise understanding of such exponential sums. 

\begin{proposition}\label{prop3.5}
  Let $f\in \Zz[X]$ be a polynomial of degree $d$, and suppose that $f$
  is Sidon--Morse over $\Qq$.  Let $K_d$ denote the compact group
  $USp_{d-1}(\Cc)$ if $f$ is symmetric Sidon--Morse, and the compact
  group $SU_{d-1} (\Cc)$ if $f$ is Sidon--Morse but not symmetric.  For
  any integer $k \ge 0$ we have
  $$ 
  \lim_{p\to+\infty}\frac{1}{p} \sum_{a\in\Ff_p^{\times}}|W(a;p)|^{2k}=
  \int_{K_d}|\Tr(g)|^{2k}d\mu(g),
  $$
  where $\mu$ is the Haar measure on $K_d$ normalized to have total
  volume $1$.  Furthermore
 \begin{equation*} 
 \int_{USp_{d-1}} |\Tr(g)|^{2k}d\mu (g) 
 \begin{cases} 
 =(2k-1)!! &\text{ for } 1\le k\le (d-1)/2\\ 
 \le (2k-1)!! &\text{ for all } k\ge 1.
 \end{cases} 
 \end{equation*} 
and 
 \begin{equation*} 
 \int_{SU_{d-1}} |\Tr(g)|^{2k}d\mu (g) 
 \begin{cases} 
 =k! &\text{ for } 0\le k\le (d-1)\\ 
 \le k! &\text{ for all } k\ge 0, 
 \end{cases} 
 \end{equation*} 
\end{proposition} 
\begin{proof} This is largely a consequence of the work of Katz~\cite{katz-esde}.  We recall Katz's work in 
Theorem~\ref{th-connect} below, and explain the link to the moments over $K_d$ in Remark~\ref{rm-35}.  Further discussion of Katz's theorem  may be found in Section~\ref{sec-remark}.   

The moments over $K_d$ for small $k$
(which match the moments of a standard complex Gaussian for $K_d= SU_{d-1}$, and the moments of a standard real Gaussian for $K_d=USp_{d-1}$) were computed by Diaconis and Shahshahani, and the upper bounds for all $k$ may be found in the work of Perret-Gentil~\cite[Prop. 2.2]{perret}.

\end{proof} 

\section{Proof of Theorem \ref{th-main}}
\label{sec-main}

We begin with the first part of the theorem, which seeks a bound for $\sum_{q\le x} |W(a;q)|^2$.   We apply
Theorem~\ref{thm5} to the function $q\mapsto W(a;q)^2$, which is
twisted-multiplicative. The Weil bound \eqref{3.1} allows us to take $G(p)=(d-1)^2$ for all but finitely
many primes.  Writing 
$$
g(p)=\frac{1}{p}\sum_{(a,p)=1}|W(a;p)|^2, 
$$
and recalling that $f$ is indecomposable, Corollary \ref{cor3.4} gives 
$$ 
\sum_{p\le x} \frac{g(p)}{p} = \log \log x +O(1). 
$$  
Theorem~\ref{thm5} yields 
$$ 
\sum_{q\le x} |W(a;q)|^2 \ll \frac{x}{\log x} \exp\Big( \sum_{p\le x}\frac{g(p)}{p}\Big) (\log \log x)^{(d-1)^2} 
\ll x  (\log \log x)^{(d-1)^2}. 
$$ 
 
Now we turn to the proof of the second part of the theorem, which we
will deduce from Theorem \ref{thm5} and Theorem \ref{th-4} (to be proved
in Section~\ref{sec-katz}).  Applying Theorem \ref{thm5} to the twisted
multiplicative function $|W(a;q)|$ and using the Weil bound (which permits $M=d-1$ here) we obtain
\begin{equation}
  \label{4.1} 
  \sum_{q\le x} |W(a;q)| \ll \frac{x}{\log x} (\log \log x)^{d-1} \exp\Big(\sum_{p\le x} \frac{1}{p} \Big(\frac 1p \sum_{a\in \Ff_p^\times}|W(a;p)|\Big)\Big). 
\end{equation}
  
Let $\epsilon$ be a small positive number, and let ${\mathcal P}$ denote the set of primes $p$ for which 
$$ 
\frac 1p \sum_{a\in \Ff_p^\times}|W(a;p)| \ge 2 -\epsilon.
$$ 
By Theorem \ref{th-4} we know that the set ${\mathcal P}$ has density
$\ge \delta =\delta(d) >0$ with $\delta$ depending only on $d$.  For any
real number $y$ with $|y| \le d-1$ we claim that
$$ 
|y| \le \frac{1+y^2}{2}, \qquad \text{and} \qquad |y|\le \frac{1+y^2}{2} + \frac{3/2-y^4}{200(d-1)^4}.
$$ 
The first inequality is clear, and so is the second inequality in the range $y^4 \le 3/2$.   In the range $3/2 < y^4 \le (d-1)^4$, note that 
$(1+y^2)/2-|y| \ge (1+\sqrt{3/2})/2 - (3/2)^{1/4} > 1/200$, so that the desired inequality holds in this case also.  

Applying the first inequality above for primes $p\notin {\mathcal P}$, we find 
$$ 
\frac 1p \sum_{a\in {\Ff_p^{\times}}} |W(a;p)| \le \frac 12 + \frac{1}{2p} \sum_{a\in \Ff_p^{\times}} |W(a;p)|^2,
$$ 
while applying the second inequality above for primes $p\in {\mathcal P}$ we find 
\begin{align*}
\frac 1p \sum_{a\in {\Ff_p^{\times}}} |W(a;p)| &\le \frac 12 + \frac{1}{2p} \sum_{a\in \Ff_p^{\times}} |W(a;p)|^2 + \frac{1}{200(d-1)^4} \Big( \frac 32- \frac 1p \sum_{a\in \Ff_p^{\times}} |W(a;p)|^4 \Big) \\
&\le \frac 12 + \frac{1}{2p} \sum_{a\in \Ff_p^{\times}} |W(a;p)|^2 -\frac{1}{400 (d-1)^4}. 
\end{align*}
Combining both inequalities, and using the first part of Corollary \ref{cor3.4}, we conclude that 
$$ 
\sum_{p\le x} \frac 1{p^2} \sum_{a\in {\Ff_p^{\times}}} |W(a;p)| \le \Big( \frac 12 + \frac 12 -\frac{\delta}{400(d-1)^4} +o(1) \Big) \log \log x. 
$$ 
Inserting this bound in \eqref{4.1}, the second part of the theorem follows.
 
\section{The fourth moment: Proof of Theorem~\ref{th-4}}\label{sec-katz}

As we shall see, for Sidon--Morse polynomials, the work of
Katz~\cite{katz-esde} can be used to show that Case (1) of
Theorem~\ref{th-4} holds.  The main challenge is to handle all
polynomials of degree~$\geq 3$, and not just the generic ones.  

Let $f\in\Zz[X]$ be a polynomial with~$d=\deg(f)\geq 3$.  If $(f(X)-f(Y))/(X-Y)$ is
not absolutely irreducible, then Lemma~\ref{lem3.2} shows that there is a positive density of primes on which the 
second moment of $W(a;p)$ is at least $2 + O(p^{-1/2})$, so that by Cauchy--Schwarz a stronger form of the second case of Theorem~\ref{th-4} holds (with the fourth moment being $\ge 4+O(p^{-1/2})$).



From now on, we will therefore assume that the polynomial
$$
F(X,Y)=f(X)-f(Y))/(X-Y)
$$
is absolutely irreducible. The remaining part of the proof will use in
an essential way the algebraic interpretation of the exponential sums
$W(a;p)$, which goes back to Weil, and it seems difficult to prove the
lower bound for the fourth moment with a direct elementary argument.

Fix a prime~$\ell$ (for instance $\ell=2$); all primes~$p$ below will be
assumed to be different from~$\ell$ and to be larger than~$d$.
Let~$\iota$ be a fixed isomorphism $\bQl\to\Cc$; we use it identify
$\ell$-adic numbers and complex numbers.

Let~$p\not=\ell$, $p>d$, be a prime number. We denote by $\psi_p$ the
$\ell$-adic additive character of~$\Ff_p$ such that
$$
\iota(\psi_p(a))=e\Bigl(\frac{a}{p}\Bigr)
$$
for~$a\in\Ff_p$.

Let~$\mathcal{G}_p$ be the $\ell$-adic sheaf $f_*\bQl/\bQl$ on the
affine line~$\Aa^1_{\Ff_p}$; it has rank $d-1$ and is everywhere tamely
ramified (since $p>d$). The sheaf $\mathcal{G}_p$ is a Fourier sheaf in
the sense of Katz (\cite[7.3.5]{katz-esde}), and we denote
by~$\mathcal{F}_p$ its (unitarily normalized) Fourier transform with
respect to~$\psi_p$ (defined in~\cite[7.3.3]{katz-esde}, up to the
normalization). The trace function of~$\mathcal{F}_p$ takes value~$0$
for $a=0$ and takes value (after applying $\iota$)
$$
\frac{1}{\sqrt{p}}\sum_{x\in\Ff_p} e\Bigl(\frac{af(x)}{p}\Bigr)=W(a;p)
$$
for~$a\in\Ff_p^{\times}$ (see~\cite[Th. 7.3.8, (4)]{katz-esde}, where
again the Fourier transform is not normalized). The rank
of~$\mathcal{F}_p$ is also equal to $d-1$, and $\mathcal{F}_p$ is
lisse and pure of weight~$0$ outside~$0$ and~$\infty$ (see~\cite[Lemma
7.3.9]{katz-esde}).


\begin{lemma}\label{lm-geom-irred}
  If the polynomial $(f(X)-f(Y))/(X-Y)$ is absolutely irreducible
  over~$\Qq$, then for all $p$ large enough, the sheaf~$\mathcal{F}_p$
  is geometrically ireducible.
\end{lemma}

\begin{proof}
  This is Fourier-side variant of Lemma~\ref{lem3.2}.  If the polynomial
  $$
  F(X,Y)=(f(X)-f(Y))/(X-Y)
  $$
  is absolutely irreducible, then the curve $C_{f,p}$ over~$\Ff_p$ with
  equation
  $$
  (f(x)-f(y))/(x-y)=0
  $$
  is geometrically irreducible, which by the Riemann Hypothesis for
  curves implies that as~$\nu\to +\infty$, we have
  $$
  |C_{f,p}(\Ff_{p^{\nu}})|\sim p^{\nu}.
  $$
  But the discrete Parseval formula implies that
  $$
  \frac{1}{p^{\nu}} |C_{f,p}(\Ff_{p^{\nu}})|= \frac{1}{p^{\nu}}
  \sum_{a\in\Ff_{p^{\nu}}^{\times}} \Bigl| \frac{1}{p^{\nu/2}}
  \sum_{x\in\Ff_{p^{\nu}}}e\Bigl(\frac{\Tr(af(x))}{p}\Bigr) \Bigr|^2
  $$
  (with the trace from $\Ff_{p^{\nu}}$ to $\Ff_p$) so we obtain
  $$
  \lim_{\nu\to+\infty} \frac{1}{p^{\nu}}
  \sum_{a\in\Ff_{p^{\nu}}^{\times}} \Bigl| \frac{1}{p^{\nu/2}}
  \sum_{x\in\Ff_{p^{\nu}}}e\Bigl(\frac{\Tr(af(x))}{p}\Bigr)
  \Bigr|^2=1,
  $$
  and this implies that~$\mathcal{F}_p$ is geometrically irreducible by
  Katz's diophantine criterion for irreducibility (see e.g.~\cite[Lemma
  4.14]{kms}).
\end{proof}

We now consider only primes $p$ such that the sheaf~$\mathcal{F}_p$ is
geometrically irreducible.

Let $G_p$ be the arithmetic monodromy group of~$\mathcal{F}_p$ and
$G_p^g$ the geometric monodromy subgroup; we can view these as algebraic
subgroups of $\GL_{d-1}(\bQl)$. The irreducibility property
of~$\mathcal{F}_p$ means that~$G_p^g$ acts irreducibly on~$\bQl^{d-1}$.

By a deep theorem of Deligne (see~\cite[Th. 3.4.1 (iii) and
Cor. 1.3.9]{deligne}), the connected component of the identity
$G_{p,0}^g$ of the group $G_p^g$ is semisimple. It is invariant under
all automorphisms of~$G_p^g$, hence it is a normal subgroup of~$G_p$
(since inner automorphisms of~$G_p$ induce automorphisms of its normal
subgroup~$G_p^g$). Let~$f_p$ denote a fixed element of the conjugacy
class of the Frobenius automorphism at~$p$.
\par
Let~$\mathcal{E}_p$ be the sheaf $\End(\End(\mathcal{F}_p))$. Its
trace function for~$a\in\Ff_p^{\times}$ is $|W(a;p)|^4$.
\par
Let $V_p$ be the subspace $\End(\End(\bQl^{d-1}))^{G_p^g}$ of vectors
invariant under $G_p^g$, the action of~$G_p$ on the space
$\End(\End(\bQl^{d-1}))$ being ``the obvious one'' induced by the
action on~$\bQl^{d-1}$ (if a group~$G$ acts on a vector space~$E$, it
acts on~$\End(E)$ by $g\cdot u=g\circ u\circ g^{-1}$).
\par
Applying the Grothendieck--Lefschetz trace formula and Deligne's
version of the Riemann Hypothesis, we get a formula
\begin{equation}\label{eq-rh}
  \frac{1}{p}\sum_{a\in\Ff_p^{\times}} |W(a;p)|^4
  =\iota(\Tr(f_p|V_p))+O(p^{-1/2})
\end{equation}
where the implied constant depends only on~$d$ (e.g. by conductor
estimates, much as in~\cite[Th. 9.1]{fkm2}).

\begin{proposition}\label{pr-motivic}
  There exists a finite Galois extension~$K$ of~$\Qq$ of degree
  bounded in terms of~$d$ only such that for all but finitely many
  primes~$p$ that are totally split in~$K$, the action of $f_p$
  on~$V_p=\End(\End(\bQl^{d-1}))^{G_p^g}$ is trivial.
\end{proposition}

Let us admit this proposition and conclude the proof of
Theorem~\ref{th-4}. For primes totally split in the number field~$K$,
we have $\iota(\Tr(f_p|V_p))=\dim(V_p)$.  On the other hand, the
definition of the action of $G_p^g$ on $\End(\bQl^{d-1})$ shows that
the space~$V_p$ is the space of all linear
maps~$\End(\bQl^{d-1})\to \End(\bQl^{d-1})$ which commute with the
$G_p^g$-action.  The identity is an element of this space, so its
dimension is $\geq 1$. Since the action on $\End(\bQl^{d-1})$ is
semisimple (e.g. by Deligne's Theorem~\cite[Th. 3.4.1]{deligne}
because it is still pure of weight~$0$), Schur's Lemma in
representation theory (see, e.g.,~\cite[Prop. 2.7.15 (3)]{repr})
implies that the dimension of~$V_p$ is exactly~$1$ if and only if the
action of $G_p^g$ on $\End(\bQl^{d-1})$ is irreducible. But $V_p$
contains both the multiples of the identity and the space
$\End^0(\bQl^{d-1})$ of matrices of trace zero as stable subspaces, so
this irreducibility can only hold if $\End^0(\bQl^{d-1})$ is zero,
i.e., if~$d=2$. So for primes totally split in~$K$, we have
$\dim(V_p)\geq 2$ hence
$$
\frac{1}{p}\sum_{a\in\Ff_p^{\times}} |W(a;p)|^4\geq 2+O(p^{-1/2})
$$
by~(\ref{eq-rh}).

To improve on this unless the limit is equal to~$2$, we use very deep
work of Katz~\cite[Th. 14.3.4]{katz-esde}
that implies that $G_{p,0}^g$ is independent of~$p$ for all~$p$ large
enough. Take a prime $p$ large enough so that $G_{p,0}^g$ has
stabilized and suppose that $\dim(V_p)=2$ for some~$p$ split
in~$K$. Then the group~$G_{p,0}^g$ must act irreducibly on matrices of
trace zero. But the Lie algebra of $G_p^g$ is a stable subspace, so
that we must have $\mathrm{Lie}(G_{p,0}^g)=\End^0(\bQl^{d-1})$. That
means that $G_{p,0}^g$ is equal to $\SL_{d-1}(\bQl)$. Then for all
primes~$p$ large enough we have $ZG_{p,0}^g=\GL_{d-1}(\bQl)$,
where~$Z$ is the group of scalar matrices in~$\GL_{d-1}(\bQl)$, which
implies that~$f_p$ acts trivially for all $p$ large enough, and then
that the limit of the fourth moments exists and is equal to~$2$.

To finally show that the constant~$2$ is best possible, we recall that
Katz has proved 
that if~$f$ is a Sidon--Morse polynomial (e.g., the derivative $f'$
has Galois group~$S_{d-1}$), then $G_{p}^g$ contains $\SL_{d-1}(\bQl)$
for all~$p$ large enough (see Theorem~\ref{th-connect}), in which case
it is well-known that the action of~$G_p^g$ on the space of matrices
of trace zero is irreducible, so that the dimension of~$V_p$ is then
equal to~$2$ for all~$p$ large enough.

\begin{remark}
  The arguments above are related to the easiest part of the
  Larsen Alternative~\cite{katz-larsen}.
\end{remark}

\begin{proof}[Proof of Proposition~\ref{pr-motivic}]
  We will begin by proving the statement without the information that
  the degree of~$K$ can be bounded in terms of~$d$ only, since the
  latter requires extra ingredients.
  \par
  \textbf{Step 1.} We first prove that, for all primes~$p$ large
  enough, the action of~$f_p$ on~$V_p$ is of finite order. Since we
  are assuming that~$G_p^g$ acts irreducibly on~$\bQl^{d-1}$, a result
  of Katz shows that the connected component of the identity
  $G_{p,0}^g$ of~$G_p^g$ acts irreducibly on~$\bQl^{d-1}$,
  provided~$p>d$ (see~\cite[Lemma 7.7.5]{katz-esde}), which we have
  assumed to be the case.
  \par
  Recall that the group of outer automorphisms of~$G_{p,0}^g$ is the
  group $\mathrm{Out}(G_{p,0}^g)$ of automorphisms modulo inner
  automorphisms.  For~$g\in G_p$, let
  $\alpha_p(g)\in \mathrm{Out}(G_{p,0}^g)$ be the class modulo inner
  automorphisms of the automorphism $x\mapsto xgx^{-1}$ of $G_{p,0}^g$
  (it is an automorphism since~$G_{p,0}^g$ is normal in~$G_p$). This
  defines a group homomorphism
  $$
  G_p\fleche{\alpha_p} \mathrm{Out}(G_{p,0}^g).
  $$
  \par
  We claim that the kernel of $\alpha_p$ is $G_{p,0}^gZ\cap G_p$
  where~$Z$ is again the group of scalar matrices
  in~$\GL_{d-1}(\bQl)$.  Indeed, the condition $\alpha_p(g)=1$ means
  that there exists~$h\in G_{p,0}^g$ such that $ gxg^{-1}=hxh^{-1}$
  for all $x\in G_{p,0}^g$, which is equivalent to $h^{-1}g$ belonging
  to the centralizer of~$G_{p,0}^g$ in~$\GL_{d-1}(\bQl)$, or in other
  words, to $h^{-1}g$ commuting with the action of $G_{p,0}^g$ on
  $\bQl^{d-1}$. By Schur's Lemma (see, e.g.,~\cite[Prop. 2.7.15
  (2)]{repr}), the irreducibility of the action of~$G_{p,0}^g$ implies
  that this centralizer is equal to~$Z$. Thus $g\in\ker(\alpha_p)$ is
  equivalent to~$g\in G_{p,0}^gZ\cap G_p$.

  We deduce therefore that we have an injective group homomorphism
  $$
  G_p/(G_{p,0}^gZ\cap G_p)\fleche{\alpha_p} \mathrm{Out}(G_{p,0}^g).
  $$
  \par
  Because $G_{p,0}^g$ is a connected semisimple group, its outer
  automorphism group is finite (see, e.g.,~\cite[p.~42,
  prop.~18]{lie9} in the case of compact groups). Hence $\alpha_p$
  injects $G_p/(G_{p,0}^gZ\cap G_p)$ in a finite group. Since $Z$ acts
  trivially on $\End(W)$ for any representation~$W$, and since~$G_p^g$
  acts trivially on~$V_p$, this shows that the order of the action
  of~$f_p$ on $V_p$ is a divisor of the order of the outer
  automorphism group.
  \par
  \textbf{Step 2.} We next prove that there exists a
  finite-dimensional continuous $\ell$-adic Galois representation
  $$
  \rho\colon \Gal(\bar{\Qq}/\Qq)\to \GL(E)
  $$
  for some $\bQl$-vector space~$E$, such that for all but finitely
  many primes, the action of Frobenius at~$p$ on~$E$ ``is'' is the
  same as the action of $f_p$ on~$V_p$.  It is enough to define a
  constructible $\ell$-adic sheaf~$\mathcal{V}$
  on~$\spec(\Zz[1/\ell N])$ for some integer~$N\geq 1$ such that the
  stalk over all but finitely many primes~$p$ ``is'' the space~$V_p$,
  and such that the action of~$f_p$ coincides with the action of the
  Frobenius at~$p$. Indeed, this sheaf~$\mathcal{V}$ will be lisse
  outside of a finite set~$S$ of primes, hence will correspond to a
  Galois representation of the Galois group of the maximal extension
  unramified outside~$S$, and this is a quotient of the Galois group
  of~$\Qq$.
  \par
  To construct~$\mathcal{V}$, we use~\cite[Lemma 4.23]{kms} (see
  also~\cite[Lemma 4.27]{kms} for a more difficult application),
  applied to the data
  \begin{multline*}
    (X,Y,f,g)=(\Aa^4,\spec(\Zz[1/\ell]),\text{the structure morphism},
    \\
    g(x,y,z,w)= f(x)+f(y)-f(z)-f(w))
  \end{multline*}
  and take the second cohomology sheaf of the complex resulting from
  this application of~\cite[Lemma 4.23]{kms}.
  
  That this ``works'' results from the expression
  $$
  \frac{1}{p}\sum_{a\in\Ff_p^{\times}}\Bigl| \frac{1}{\sqrt{p}}
  \sum_{x\in \Ff_p}e\Bigl(\frac{af(x)}{p}\Bigr) \Bigr|^4=
  \frac{1}{p^3}\sum_{x,y,z,w\in\Ff_p}\sum_{a\in\Ff_p^{\times}}
  e\Bigl(\frac{ag(x,y,z,w)}{p}\Bigr),
  $$
  combined with the cohomological expression
  \begin{equation}\label{eq-iso}
    V_p\simeq H^2_c(\Gg_m\times \bFp, \End(\End(\mathcal{F}_p)))(1).
  \end{equation}
  \par
  \textbf{Step 3.} By the compatibility with Frobenius of the
  isomorphism~(\ref{eq-iso}) in Step~2, and by Step~1, the action of
  Frobenius at~$p$ under~$\rho$ is of finite order for all but finitely
  many primes~$p$. The image~$H$ of~$\rho$ is a compact $\ell$-adic Lie
  group (identifying~$\GL(E)$ with~$\GL_m(\bQl)$ for some~$m\geq 1$, we
  first note that~$H$ is contained in~$\GL_m(L)$ for some finite
  extension~$L$ of~$\Qq_{\ell}$, by an oft-rediscovered lemma -- see for
  instance~\cite[Lemma 9.0.8]{katz-sarnak} -- and then it is a closed
  subgroup of an $\ell$-adic Lie group, hence itself an~$\ell$-adic Lie
  group by, e.g.,~\cite[p. 227, th.~2]{lie3}). It follows
  from~\cite[Cor. 1, p. 169]{lie3} that there is a neighborhood~$U$
  of~$1\in H$ which contains no non-trivial finite subgroup; there is then
  a number field~$K$ such that the finite-index subgroup
  $\Gal(\bar{\Qq}/K)$ maps to~$U$. All the Frobenius elements in this
  subgroup (which exist outside any given finite set of primes because
  Frobenius elements are dense, by a form of Chebotarev's density theorem)
  must map to the identity, which means that~$\Gal(\bar{\Qq}/K)$ is in the
  kernel of~$\rho$. This implies that for a prime~$p$ that is totally split
  in~$K$, the action of~$f_p$, which ``is'' the action of Frobenius
  under~$\rho$, is trivial. This proves the result, up to the bound on the
  degree of~$K$.
  \par
  \textbf{Step 4.} Now we explain how to bound the degree of~$K$ in
  terms of~$d$ only.

  The first ingredient is a fact from the theory of finite groups: for
  given positive integers $k$ and $m$, if $\Gamma$ is a finite
  subgroup of~$\GL_k(\bQl)$ such that all elements of~$\Gamma$ have
  order dividing~$m$, then the order of~$\Gamma$ is bounded in terms
  of~$k$ and~$m$ only. Indeed, by a well-known theorem of Jordan (see,
  e.g.,~\cite[Th. 36.13]{cr}), there exists a normal abelian
  subgroup~$\Gamma_0$ of $\Gamma$ of index bounded in terms of $k$ and
  $m$. This reduces the problem to the abelian case; but $\Gamma_0$
  can be diagonalized, and the bound on the order of its elements show
  that $\Gamma_0$ is isomorphic to a subgroup of $(\Zz/m\Zz)^k$, hence
  the result.
  \par
  We want to apply this to the image~$\Gamma\subset \GL(E)$ of the
  Galois representation~$\rho$. We have~$\dim(E)\leq (d-1)^4$. By the
  Chebotarev Density Theorem, it is then enough to prove that the
  order of the action of~$f_p$ on~$V_p$ is uniformly bounded in terms
  of~$d$ only. For this we use the fact that there are, up to
  isomorphism, only finitely many possibilities for~$G_{p,0}^g$, since
  it is a connected and semisimple subgroup of~$\GL_{d-1}$ (this
  follows, in the equivalent case of compact Lie groups, from the
  discussion in~\cite[\S 4, $\text{n}^0$9, Scholie]{lie9}, which shows
  that such subgroups are classified by their root system~$R$, which
  here has rank~$\leq d-1$, which gives only finitely many
  possibilities, and for each root system~$R$ by a subgroup of the
  quotient~$Q(R)/P(R)$ discussed in loc. cit.; since this quotient is
  finite by~\cite[\S 1, $\text{n}^o$10]{lie6}, there are again only
  finitely many possibilities).  So the order of~$f_p$ is a divisor of
  the order of one of finitely many finite groups (depending only
  on~$d$).
\end{proof}

\section{Generic polynomials}\label{sec-generic}

In this section, we will prove the kind of non-correlation estimates
modulo primes that are needed in the proof of
Theorem~\ref{th-main2}. We also explain Proposition~\ref{prop3.5} at
the end.

We first make some remarks concerning Sidon--Morse polynomials:

\begin{lemma}\label{lm-indec}
  Let $K$ be any field and let $f\in K[X]$ be a Morse polynomial of
  degree~$d\geq 2$.
  \par
  \emph{(1)} The polynomial $f$ is indecomposable over~$K$.
  \par
  \emph{(2)} For any $c\in K$, the polynomials $f+c$ and $-f+c$ are
  Morse polynomials. If $f$ is a Sidon--Morse polynomial, then $f+c$ and
  $-f+c$ are Sidon--Morse polynomials.
\end{lemma}

\begin{proof}
  (1) We show that if~$f$ is decomposable, then it is not a Morse
  polynomial. Let $f=g\circ h$ where $\deg(g)\geq 2$ and~$\deg(h)\geq 2$
  be a decomposable polynomial. Note that $p$ does not divide either
  $\deg(g)$ or~$\deg(f)$ since $p\nmid d$.
  \par
  For any critical point $\alpha$ of~$g$, the critical values of~$f$
  contain, with multiplicity,
  the values $g(h(\beta))$ where $h(\beta)=\alpha$. This will give
  rise to a critical value with multiplicity at least~$2$
  \emph{unless} $h-\alpha=\gamma( X-\beta)^{\deg(h)}$ for some
  $\gamma\in K^{\times}$. Since $p$ does not divide~$\deg(h)$, this
  can only occur for a single value of~$\alpha$, so that~$g$ is of the
  form
  $$
  g=\delta (X-\alpha)^{\deg(g)}+\eta
  $$
  for some $\delta\in K^{\times}$ and $\eta\in K$. Then we get
  $$
  g\circ h=\eta + \delta \gamma^{\deg(g)}(X-\beta)^d,
  $$
  which has a single critical value, and is therefore not a Morse
  polynomial.
  \par
  (2) This is straightforward from the definition, since the critical
  points of $g=f+c$ (resp. $g=-f+c$) are the same as those of~$f$, so
  the critical values of~$g$ are those of $f$ translated by~$c$
  (resp. the negative of those of~$f$, translated by~$c$).
\end{proof}

Let~$p$ be a prime number and~$f\in\Ff_p[X]$ a Sidon--Morse
polynomial. We define the $\ell$-adic sheaf~$\mcF_f$ associated to~$f$
as in the previous section.  We will now normalize it in a specific
way. We denote by $\mcL_2$ the Kummer sheaf associated to the Legendre
character, with trace function $a\mapsto (a/p)$.

\begin{definition}[Normalized sheaf]
  Let $p$ be a prime and $f\in\Ff_p[X]$ a Sidon--Morse polynomial with
  $p\nmid \deg(f)-1$.
  \par
  (1) If $f$ is not symmetric Sidon, then there is a unique $c\in \Ff_p$
  such that the sum of the critical values of $f+c$ is equal to~$0$, and
  the \emph{normalized} $\mctF_f$ sheaf of $f$ is defined to be
  $$
  \mctF_f=\mcF_{f+c}\otimes \mathcal{L}_2^{d-1}.
  $$
  \par
  We then say that~$c=c_f$ is the \emph{critical shift} of~$f$.
  \par
  (2) If $f$ is symmetric Sidon polynomial, and
  $$
  f=g(\beta X+\gamma)+\delta
  $$
  where~$g$ is odd, then we put
  $$
  \mctF_f=\mcF_g.
  $$
  \par
  We note that the sum of critical values of~$g$ is then equal to~$0$.
\end{definition}

The trace function of $\mctF_f$ is $0$ for~$a=0$ and for
$a\in\Ff_p^{\times}$ is given either by
\begin{equation}\label{eq-trace-norm}
  \widetilde{W}_f(a;p)=\frac{1}{\sqrt{p}}\Bigl(\frac{a}{p}\Bigr)^{d-1}
  \sum_{x\in\Ff_p}e\Bigl(\frac{a(f(x)+c)}{p}\Bigr)=
  \Bigl(\frac{a}{p}\Bigr)^{d-1}e\Bigl(\frac{ac}{p}\Bigr)
  \sum_{x\in\Ff_p}e\Bigl(\frac{af(x)}{p}\Bigr)
\end{equation}
or by
\begin{multline}\label{eq-trace-norm2}
  \widetilde{W}_f(a;p)= \frac{1}{\sqrt{p}}
  \sum_{x\in\Ff_p}e\Bigl(\frac{ag(x)}{p}\Bigr)
  \\=
  \frac{1}{\sqrt{p}}
  e\Bigl(-\frac{a\delta}{p}\Bigr)
  \sum_{x\in\Ff_p}e\Bigl(\frac{af((x-\gamma)/\beta)}{p}\Bigr)
  =
  e\Bigl(-\frac{a\delta}{p}\Bigr)W_f(a;p).
\end{multline}
in the symmetric case. In particular, we see that in all cases, the
formula
$$
|\widetilde{W}_f(a;p)|=|W_f(a;p)|
$$
is valid all~$a$ modulo~$p$.

The point of this normalization is the following theorem of Katz:

\begin{theorem}[Katz]\label{th-connect}
  Let~$p$ be a prime number. Let $f\in\Ff_p[X]$ be a Sidon--Morse
  polynomial of degree $d\geq 3$. Assume that $p>2d-1$ and
  that~$p\nmid d-1$.
  \par
  \emph{(1)} If $f$ is not a symmetric Sidon--Morse polynomial, then the
  geometric monodromy group of $\mctF_f$ is equal to $\SL_{d-1}(\bQl)$.
  \par
  \emph{(2)} If $f$ is a symmetric Sidon--Morse polynomial, which
  implies that $d$ is odd, then the geometric monodromy group of
  $\mctF_f$ is isomorphic to $\Sp_{d-1}(\bQl)$.
\end{theorem}

\begin{proof}
  (1) If $f$ is not of symmetric type, then the geometric monodromy
  group contains $\SL_{d-1}$ under the assumption on $p$,
  by~\cite[Th. 7.9.6]{katz-esde}, and has trivial determinant
  by~\cite[Lemma 7.10.4, (2)]{katz-esde}, so it must be $\SL_{d-1}$.
  \par
  (2) If $f$ is of symmetric type, then under the assumption on $p$, a
  conjugate of the geometric monodromy group of $\mctF_f$ is contained
  in $\Sp_{d-1}$ by~\cite[Lemma 7.10.4, (3)]{katz-esde} (since the
  associated polynomial~$g$ is odd). By~\cite[Th. 7.9.7]{katz-esde}, it
  contains either $\SL_{d-1}$ or $\Sp_{d-1}$ or $\SO_{d-1}$; the only
  possibility that is compatible with both these facts is that it is
  $\Sp_{d-1}$.
\end{proof}

\begin{remark}
  If we consider a Morse polynomial~$f$ such that the set of critical
  values is a symmetric Sidon set, we might hope that (2) still holds.
  However, although one can still deduce from the work of Katz that the
  geometric monodromy group of $\mctF_f$ contains a symplectic group, we
  currently do not know if this condition is sufficient to ensure that
  $\mctF$ has conversely a symplectic symmetry.
\end{remark}

We will also need a result that is essentially a consequence of the ideas
of Fried.

\begin{proposition}\label{pr-cn-c}
  Let $p$ be a prime number.  Let~$f$ and~$g$ in~$\Ff_p[X]$ be
  Sidon--Morse polynomials of respective degree~$d_f\geq 3$
  and~$d_g\geq 3$. Assume that $d_f<p$ and~$d_g<p$.
  \par
  If $f$ and~$g$ are not linearly equivalent over~$\bar{\Ff}_p$, then
  $f(X)-g(Y)+c$ and $f(X)+g(Y)+c$ are absolutely irreducible for
  any~$c$.
\end{proposition}

\begin{proof}
  Since~$f+c$ is a Sidon--Morse polynomial (Lemma~\ref{lm-indec}), and
  linearly equivalent to~$g$ if and only $f$ is, we can assume
  that~$c=0$. Since $-g$ is a Sidon--Morse polynomial, and linearly
  equivalent to~$f$ if and only if so is~$g$, we need only consider
  the case of $f(X)-g(Y)$.

  Let~$G$ be the Galois group of the equation $f(X)-Y=0$ over the
  field $\bar{\Ff}_p(Y)$ (so~$X$ is the variable). If $f(X)-g(Y)$ is
  not absolutely irreducible then $G$ is also isomorphic to the one
  for the equation $g(X)-Y=0$ by~\cite[\S 2.1.1]{cn-c}.\footnote{\
    This is written for the base field~$\Cc$, but the argument extends
    to any algebraically closed field when the polynomials involved
    have degree less than the characteristic of the field.}
  By~\cite[\S 2.1.4]{cn-c}, if $f$ and $g$ are not linearly equivalent
  over~$\bar{\Ff}_p$, then the faithful permutation representations
  of~$G$ on the roots of these two equations are not equivalent as
  permutation representations, but have the same character (i.e., are
  equivalent as linear representations). However, for Sidon--Morse
  polynomials $f$ and $g$, the group $G$ and its permutation
  representation are isomorphic to $\mathfrak{S}_{d}$ with the
  standard permutation representation on $d$ letters (see~\cite[Proof
  of Lemma 7.10.2.3]{katz-esde}).  But this is a contradiction, since
  this faithful permutation representation of~$\mathfrak{S}_d$ is
  characterized by its character (the only non-obvious case is
  when~$d=6$ and we consider the standard permutation representation
  and that given by a non-trivial outer automorphism
  of~$\mathfrak{S}_6$, but these have different characters,
  e.g. because a transposition is mapped to, respectively, a
  transposition, with $4$ fixed points, or a product of three disjoint
  transpositions, without fixed points).
\end{proof}

\begin{proposition}\label{pr-mult}
  Let $p$ be a prime.  Let $m\geq 1$ be an integer and let $f_1$,
  \ldots, $f_m$ be Sidon--Morse polynomials in $\Ff_p[X]$. Assume that
  $p>2\deg(f_i)-1$ and~$p\nmid (\deg(f_i)-1)$ for all $i$. Assume also
  that for all $i\not=j$, the polynomials~$f_i$ and~$f_j$ are not
  linearly equivalent over~$\bar{\Ff}_p$.
  \par
  Then the geometric monodromy group of the sheaf
  $$
  \bigoplus_{1\leq i\leq m}\mctF_{f_i}
  $$
  is the direct product of the geometric monodromy groups of the sheaves
  $\mctF_{f_i}$.
\end{proposition}

\begin{proof}
  We write $d_i=\deg(f_i)$ and $\mctF_i=\mctF_{f_i}$.
  We also denote by $\mctF_i^{\vee}$ the dual of $\mctF_i$.
  \par
  We will apply the Goursat--Kolchin--Ribet Criterion, as developed by
  Katz~\cite[Prop. 1.8.2]{katz-esde}, and expounded by Fouvry,
  Kowalski and Michel~\cite[Lemma 2.4]{sop}. In the language of
  loc. cit., it suffices to check that the family $(\mctF_{i})$ is
  $\Gg_m$-generous (\cite[Def. 2.1]{sop}), since the individual
  geometric monodromy groups of~$\mctF_i$ are connected by
  Theorem~\ref{th-connect}.
  \par
  This desired property is the combination of four conditions.
  Condition (1) holds because the sheaves $\mctF_{i}$ are pure of
  weight $0$ on $\Gg_m$, and have a geometric monodromy group (namely
  $\SL_{d_i-1}$ or $\Sp_{d_i-1}$ by Theorem~\ref{th-connect}) that
  acts irreducibly on $\bQl^{d_i-1}$.  Conditions (2) and (3) are then
  known properties of $\SL_{d_i-1}$ and $\Sp_{d_i-1}$ (see~\cite[\S
  3.1]{sop}).
  \par
  To prove the most important Condition (4), it is enough to check that if
  $i\not=j$, there is no geometric isomorphism
  \begin{equation}\label{eq-putative}
    \mctF_i\simeq \mctF_j\otimes \mcL,\quad \text{or} \quad
    \mctF_i^{\vee}\simeq \mctF_j\otimes \mcL
  \end{equation}
  where $\mcL$ is a rank one sheaf lisse on $\Gg_m$ (see~\cite[Remark
  2.2]{sop}). This is impossible unless $d_i=d_j$ and unless either none
  or both of~$f_i$ and~$f_j$ are symmetric Sidon--Morse.
  We now assume that $d_i=d_j$ and we denote by $d$ this common value.
  \par
  \textbf{Case (1)}.  Assume first that neither $f_i$ nor $f_j$ is
  symmetric, and that we have the isomorphism
  $ \mctF_i\simeq \mctF_j\otimes \mcL$ in~(\ref{eq-putative}). We denote
  by $c_i$ and $c_j$ the critical shifts of~$f_i$ and~$f_j$.
  \par
  We recall that since $p>2d-1$, the sheaf $\mctF_i$ is, for all~$i$,
  tamely ramified at~$0$ (\cite[Lemma 7.10.4, (1)]{katz-esde}), with
  local monodromy isomorphic to the sum of the non-trivial characters of
  order $d$ (\cite[Lemma 7.10.4, (1)]{katz-esde}). These must be
  permuted by multiplication by the monodromy character $\chi_0$ of
  $\mcL$ at~$0$, which is only possible if $\chi_0=1$, i.e., if $\mcL$
  is lisse at~$0$.
  \par
  Next, by~\cite[Th. 7.8.4, (2)]{katz-esde} and the construction of
  $\mctF_i$, the wild monodromy representation of $\mctF_i$ at
  $\infty$ is the direct sum
  \begin{equation}\label{eq-local-infinity}
    \bigoplus_{v\in V_i}\mcL_{\psi(vX)}
  \end{equation}
  where $V_i$ is the set of critical values of $f_i+c_i$, and
  $\mcL_{\psi(vX)}$ denotes the Artin--Schreier sheaf modulo~$p$ with
  trace function $a\mapsto e(av/p)$. Let $v\in V_i$. The putative
  isomorphism $ \mctF_i\simeq \mctF_j\otimes \mcL$ implies that there
  exists $w\in V_j$ such that
  $$
  \mcL_{\psi(vX)}=\mcL\otimes \mcL_{\psi(wX)},
  $$
  as representations of the wild inertia group at $\infty$. In
  particular, $\mcL$ is an Artin--Schreier sheaf at infinity, say
  $\mcL\simeq \mcL_{\psi(cX)}$ for some $c$, as representations of the
  wild inertia group. The local isomorphism becomes
  $$
  \bigoplus_{v\in V_i}\mcL_{\psi(vX)}\simeq \bigoplus_{w\in
    V_j}\mcL_{\psi((c+w)X)},
  $$
  so that $V_i=V_j+c$ as subsets of $\bFp$. But taking the sum of the
  values on both sides, and using the definition of the normalized sheaf,
  we deduce that $c=0$. Thus the sheaf $\mcL$ is trivial on the wild
  monodromy group, and therefore is also tamely ramified at $\infty$.
  \par
  Since $\mcL$ is lisse on $\Gg_m$ and tame, it is a Kummer sheaf
  attached to some multiplicative character $\chi$ of~$\Ff_p^{\times}$
  (which is its trace function). Since it is lisse at~$0$, this
  character must be trivial. Hence we deduce that $\mctF_i$ and
  $\mctF_j$ are in fact geometrically isomorphic.
  \par
  By the Diophantine Criterion for Irreducibility (see
  e.g.~\cite[Lemma 4.14]{kms}), this implies that
  \begin{equation}\label{eq-irred-diophantine}
    \limsup_{\nu\to+\infty}\frac{1}{p^{\nu}}
    \Bigl|\sum_{a\in\Ff_{p^{\nu}}^{\times}}
    \widetilde{W}_i(a;p^{\nu})\overline{\widetilde{W}_j(a;p^{\nu})}\Bigr|=
    \limsup_{\nu\to+\infty}\frac{1}{p^{\nu}}\sum_{a\in\Ff_{p^{\nu}}^{\times}}
    |\widetilde{W}_i(a;p^{\nu})|^2=1,
  \end{equation}
  where $\widetilde{W}_i(a;p^{\nu})$ is the trace function of
  $\mctF_i$ over the extension of degree $\nu$ of~$\Ff_p$.
  By~(\ref{eq-trace-norm}) and orthogonality of characters, the sum on
  the left-hand side is equal to
  $$
  \frac{1}{p^{\nu}}|\{(x,y)\in\Ff_{p^{\nu}}^2\,\mid\,
  f_i(x)+c_i=f_j(y)+c_j\}|-1
  $$
  (noting that if the trace function of $f_i$ has the Legendre factor,
  then so does $f_j$, and they cancel out). If the polynomial
  $f_i(X)-f_j(Y)+c_i-c_j$ is absolutely irreducible, then we get
  $$
  \frac{1}{p^{\nu}}|\{(x,y)\in\Ff_{p^{\nu}}^2\,\mid\,
  f_i(x)+c_i=f_j(y)+c_j\}|-1
  \ll p^{-\nu/2}
  $$
  by the Riemann Hypothesis for curves, which
  contradicts~(\ref{eq-irred-diophantine}).
  Thus the polynomial
  $$
  f_i(X)-f_j(Y)+c_i-c_j
  $$
  is not absolutely irreducible, which can only happen if~$f_i$
  and~$f_j$ are linearly equivalent over~$\bar{\Ff}_p$
  (Proposition~\ref{pr-cn-c}).
  \par
  \textbf{Case 2.} We continue assuming that neither $f_i$ nor $f_j$ is
  symmetric, and consider the second case of an hypothetical
  isomorphism~(\ref{eq-putative}). It is elementary that the dual
  $\mctF_i^{\vee}$ is the normalized sheaf associated to $-f_i$ (because
  $\mcF_i$ is the Fourier transform of a sheaf that is self-dual, being
  the direct image of the self-dual constant sheaf; see~\cite[Th. 7.3.8,
  (2)]{katz-esde}). Thus we are reduced to the previous case.
  \par
  \textbf{Case 3.} Now we assume that $f_i$ and $f_j$ are
  symmetric. Since $\mctF_i$ and $\mctF_j$ are then self-dual by
  Theorem~\ref{th-connect} (2), we need only exclude the possibility
  of a geometric isomorphism of the form
  $$
  \mctF_i\simeq \mctF_j\otimes \mcL.
  $$
  Assume there is such an isomorphism.  We denote by $g_i$ and $g_j$ the
  odd polynomials associated to~$f_i$ and~$f_j$ so that
  $\mctF_i=\mcF_{g_i}$ and $\mctF_j=\mcF_{g_j}$. Arguing exactly as in
  Case 1, we see that the sheaf~$\mcL$ is trivial. Then continuing again
  as in Case 1 using~(\ref{eq-trace-norm2}) we find that there are
  $\delta_i$ and $\delta_j$ such that
  $$
  f_i(X)-f_j(Y)-\delta_i+\delta_j
  $$
  is not absolutely irreducible, and Proposition~\ref{pr-cn-c} allows us
  to conclude that $f_i$ and $f_j$ would have to be linearly dependent.
\end{proof}

\begin{lemma}\label{lm-lin-eq}
  Let $f$ and $g$ in $\Zz[X]$ be polynomials of common
  degree~$d\geq 3$.  The polynomials $f$ and~$g$ are linearly
  equivalent over~$\bar{\Qq}$ if and only if $f\mods{p}$
  and~$g\mods{p}$ are linearly equivalent over an algebraic closure
  $\bar{\Ff}_p$ of~$\Ff_p$ for infinitely many primes.
\end{lemma}

\begin{proof}
  The set $X_{f,g}$ of tuples $(a,b,c,d)$ in~$\bar{\Qq}$ such that
  $$
  g=af(cX+d)+b
  $$
  is defined by polynomial equations with rational coefficients. The
  polynomials~$f$ and~$g$ are linearly equivalent over~$\bar{\Qq}$ if
  and only if~$X_{f,g}(\bar{\Qq})$ is not empty. Since $X_{f,g}$ is an
  algebraic variety, this is true if and only
  if~$X_{f,g}(\bar{\Ff}_p)$ is not empty for all~$p$ large enough
  (e.g, by the Nullstellensatz: if $X_{f,g}(\bar{Q})$ is empty, then
  there is a representation of~$1$ as belonging to the ideal generated
  by the equations of~$X_{f,g}$, and this leads to a representation
  of~$1$ over~$\bar{\Ff}_p$ for all primes large enough), which proves
  the assertion.
\end{proof}

\begin{corollary}\label{cor-mult}
  Let $m\geq 1$ be an integer and let $f_1$, \ldots, $f_m$ be
  Sidon--Morse polynomials in $\Zz[X]$ that are pairwise not linearly
  equivalent over~$\bar{\Qq}$. Let~$s\leq m$ be the number of $f_i$ such
  that $f_i$ is symmetric Sidon--Morse of degree~$\geq 5$.
  \par
  We have
  \begin{gather}
    \frac{1}{p}\sum_{(a,p)=1}|W_{f_1}(a;p) \cdots W_{f_m}(a;p)|^2=
    1+O(p^{-1/2})\label{eq-1}
    \\
    \frac{1}{p}\sum_{(a,p)=1}|W_{f_1}(a;p) \cdots W_{f_m}(a;p)|^4=
    2^{m-s}3^s+O(p^{-1/2})
    \label{eq-2}
  \end{gather}
  where the implied constant depends only on $m$ and on the degrees of
  the polynomials $f_i$.  
\end{corollary}

\begin{proof}
  Applying Lemma~\ref{lm-lin-eq}, we see that for $p$ large enough,
  the assumptions of Proposition~\ref{pr-mult} hold modulo~$p$.
  Let~$p$ be such a prime. Using the same notation as
  in~(\ref{eq-rh}), the left-hand side of~(\ref{eq-1}) is equal to
  $$
  \iota(\Tr(f_p|\End(W_p)^G))+O(p^{-1/2})
  $$
  where $W_p$ is the tensor product space
  $$
  \bigotimes_{i} \bQl^{d_i-1}
  $$
  as a representation of the geometric monodromy group~$G$ of
  $$
  \bigoplus_{i} \mctF_i.
  $$
  By Proposition~\ref{pr-mult}, this representation can be identified
  with the external tensor product of the representations of the
  individual geometric monodromy groups; since this external tensor
  produt is an irreducible representation (see,
  e.g.,~\cite[Prop. 2.3.23]{repr}), the invariant space has dimension
  one, spanned by the scalar matrices, on which $f_p$ acts trivially,
  and the first result follows.
  \par
  For the second result, we get in the same way the main term
  of~(\ref{eq-2}) equal to
  $$
  \prod_{i=1}^m\dim (\End(\End(\bQl^{d_i-1}))^{G_i}
  $$
  where $G_i$ is the geometric monodromy group of~$\mctF_i$. By the
  simplest case of the Larsen Alternative
  (see~\cite[Th.\,1.1.6]{katz-larsen}), each factor is equal to~$3$ if
  $f_i$ is a symmetric Sidon--Morse polynomial of degree~$\geq 5$ (with
  symplectic monodromy) and to~$2$ for the others.
\end{proof}

We conclude this section with the following proposition will is used
in the proof of the first part of Theorem~\ref{th-main2}, where only
one polynomial is assumed to be a Sidon--Morse polynomial.

\begin{proposition}\label{pr-noncorr}
  Let~$f$ and $g$ be non-constant polynomials in~$\Zz[X]$ of degrees
  $d_f$ and~$d_g$ respectively. Suppose that $f$ is a Sidon--Morse
  polynomial, that $d_f<d_g$ and that $g$ is absolutely
  irreducible. Then
  $$
  \frac{1}{p}\sum_{(a,p)=1}|W_f(a;p)\overline{W_g(a;p)}|^2= 1+O(p^{-1/2})
  $$
  where the implied constant depends only on~$d_f$ and~$d_g$.
\end{proposition}

\begin{proof}
  This is a variant of the Goursat--Kolchin--Ribet argument, but where
  we only fully control one of the sheaves.
  \par
  Let $p>d_f-1$ be a prime such that $f$ is a Sidon--Morse polynomial
  modulo~$p$. We denote by $\mctF_f$ the normalized sheaf associated
  to~$f$ modulo~$p$, and by~$G_{f}$ (resp. $G_{g}$) the geometric
  monodromy group of $\mctF$ (resp.~ of~$\mcF_g$).  Since $f$ is a
  Sidon--Morse polynomial, we have $G_f=\SL_{d_f-1}$ or
  $G_f=\Sp_{d_f-1}$ (the latter when $f$ is symmetric Sidon--Morse) by
  Theorem~\ref{th-connect}.
  \par
  Let further $H$ be the geometric monodromy group of
  $\mctF_f\oplus\mathcal{F}_g$.  We have a natural inclusion
  $H\to G_{f}\times G_{g}$, and the composition of this inclusion with
  either projection is surjective.
  \par
  We denote by $W_p$ the space
  $$
  \End(\mctF_f\otimes \mathcal{F}_g)^{H},
  $$
  and by~$f_p$ a representative of the Frobenius automorphism in~$H$.
  The analogue of~(\ref{eq-rh}) in this case is the formula
  $$
  \frac{1}{p}\sum_{a\in\Ff_p^{\times}} |W_f(a;p)\overline{W_g(a;p)}|^2
  =\iota(\Tr(f_p|W_p))+O(p^{-1/2})
  $$
  where the implied constant depends only on~$d_f$ and~$d_g$ (and we used
  the fact that the trace function of $\mctF_f$ has the same modulus as
  that of $\mcF_f$).  By Schur's Lemma, it then suffices to prove that the
  representation of~$H$ on $\mctF_f\otimes \mathcal{F}_g$ is
  irreducible, and in turn it is enough to prove that $H=G_f\times G_g$
  (using again 
  the irreducibility of external tensor product of irreducible
  representations, see~\cite[Prop. 2.3.23]{repr}).
  \par
  We denote by $L$ the kernel of the composition homomorphism
  $$
  G_{f}\to H\subset G_f\times G_g\to G_{g}.
  $$
  This is a normal subgroup of $G_{f}$, hence $L$ is either finite or
  equal to $G_f$.  If the latter holds, then $H$ contains
  $G_f\times \{1\}$, and it follows easily that $H=G_f\times G_f$.
  \par
  Thus we need to exclude the possibility that $L$ is finite. However,
  if that is the case, then $G_{f}/L$ is isomorphic to a subgroup of
  $G_{g}$, hence the Lie algebra of~$G_f$ has a faithful
  representation of dimension~$\leq d_g-1$. Since we assumed that
  $d_f>d_g$, this is impossible in view of the minimal dimensions of
  faithful representations of the Lie algebras of~$\SL_{d_f-1}$
  or~$\Sp_{d_f-1}$ (which are equal to~$d_f-1$, see
  e.g.~\cite[p.\,249, Exercice 2 et p.\,214, Table\, 2]{lie8}).
\end{proof}

\begin{remark}\label{rm-35}
  Theorem~\ref{th-connect} also implies
  Proposition~\ref{prop3.5}. Indeed, using the same notation as
  in~(\ref{eq-rh}), the Riemann Hypothesis and conductor estimates
  imply that for~$k$ fixed and~$p$ large, we have
  $$
  \sum_{a\in\Ff_p^{\times}}|W(a;p)|^{2k}=
  \nu_k+O(p^{-1/2}),
  $$
  where~$\nu_k$ is the multiplicity of the trivial representation of
  the geometric monodromy group in the
  representation~$\End(\bQl^{d-1})^{\otimes k}$. By character theory
  for compact groups, we have
  $$
  \nu_k=\int_{K_d}|\Tr(g)|^{2k}d\mu(g)
  $$
  for a maximal compact subgroup~$K_d$ of the geometric monodromy
  group, where $\mu$ is the Haar measure on $K_d$ normalized to have
  total volume $1$. We can take $K_d=\SU_{d-1}(\Cc)$ if the geometric
  monodromy group is~$\SL_{d-1}$, and $K_d=\USp_{d-1}(\Cc)$ if it
  is~$\Sp_{d-1}$. 
\end{remark}

\section{Multiple correlations}\label{sec-mult}

We now come to Theorem~\ref{th-main2}.  For the first part, we apply
Theorem~\ref{thm5} to the function
$a\mapsto |W_f(a;q)\overline{W_g(a;q)}|^2$. We can take
$M=(d_f-1)^2(d_g-1)^2$.  By Proposition~\ref{pr-noncorr}, we have
$$
\frac{1}{p}\sum_{(a,p)=1}|W_f(a;p)\overline{W_g(a;p)}|^2= 1+O(p^{-1/2})
$$
so we can take~$g(p)=1+O(p^{-1/2})$. Thus Theorem~\ref{thm5} gives, for
some constant~$C\geq 0$, the bound
\begin{align*}
  \sum_{q\leq x}|W_f(a;q)W_g(a;q)|^2
  &\ll \frac{x}{\log x}\prod_{p\leq
    x}\Bigl(1+\frac{1}{p}+\frac{C}{p^{3/2}}\Bigr)
  (\log \log x)^{(d_f-1)^2(d_g-1)^2} 
  \\
  &\ll x(\log \log x)^{(d_f-1)^2(d_g-1)^2}.
\end{align*}

For the second part, we apply Theorem~\ref{thm5} to the functions
\begin{gather*}
  a\mapsto |W_1(a;q)\cdots W_m(a;q)|,\\
  a\mapsto |W_1(a;q)\cdots W_m(a;q)|^2,\\
  a\mapsto |W_1(a;q)\cdots W_m(a;q)|^4
\end{gather*}
and argue as in the proof of Theorem~\ref{th-main} using
Corollary~\ref{cor-mult}.



\section{Remarks on Katz's Theorem}
\label{sec-remark}

We want to observe that Katz's Theorem (Theorem~\ref{th-connect}) can be
explained, in the case of monodromy $\SL_{d-1}$, as the combination of
two facts:
\begin{enumerate}
\item the local monodromy computation~(\ref{eq-local-infinity}), which
  has an intuitive meaning as the algebraic analogue of the stationary
  phase expansion for oscillatory integrals
  $$
  g(t)=\int e^{itf(x)}dx,
  $$
\item a result of Gabber (see~\cite[Th.~1.0]{katz-esde}) which
  (essentially) deduces the nature of the monodromy group from the Sidon
  property of the critical values.
\end{enumerate}

Since the proof of Gabber's result, in this special case, is relatively
accessible and (in our opinion) quite enlightening with respect to the
relevance of the Sidon condition, we include the precise statement and
its proof.

\begin{proposition}\label{pr-gabber}
  Let $V$ be a finite-dimensional complex vector space of dimension
  $r\geq 1$, and let $G$ be a connected semisimple compact subgroup of
  $\GL(V)$ which acts irreducibly on~$V$. Let $D$ be the subgroup of
  elements of~$\GL(V)$ which are diagonal with respect to some basis,
  and let $\chi_i$, for $1\leq i\leq r$, be the characters
  $D\to \Cc^{\times}$ giving the coefficients of the elements of~$D$.
  \par
  Let $A\subset D$ be a subgroup of the normalizer of $G$ in
  $\GL(V)$. Let $S\subset \widehat{A}$ be the subset of the group of
  characters of~$A$ given by the restrictions to~$A$ of the diagonal
  characters $\chi_i$. If $|S|=r$ and $S$ is a Sidon set in
  $\widehat{A}$, then $G=\SU(V)$.
\end{proposition}

\begin{proof}
  We denote by $Z\subset D$ the subgroup of scalar matrices. We may
  assume that $G\subset \Un(V)$.
  
  The group~$G$ is a compact real Lie group.  We consider the
  representation of~$A$ on $\End(V)$ by conjugation. It acts on the
  elementary matrices $E_{i,j}$ by $\chi_i\chi_j^{-1}$. The assumption
  that~$S$ has $r$ elements and is a Sidon set means then that
  $$
  \End(V)=\bigoplus_{i,j} \Cc E_{i,j}
  $$
  is a decomposition of the representation as a sum of characters where,
  for $i\not=j$, the line $\Cc E_{i,j}$ is a non-trivial character of
  multiplicity one.
  \par
  Since $A\subset N_{\GL(V)}(G)$, the complexified Lie algebra
  $L\subset \End(V)$ of $G$ is a subrepresentation of the representation
  of~$A$ on~$\End(V)$.
  Thus there exists a subspace~$H$ of the diagonal matrices, and a
  subset~$X$ of pairs $(i,j)$ of distinct integers such that
  $$
  L=H\oplus\bigoplus_{(i,j)\in X}\Cc E_{i,j}.
  $$
  This implies that~$L$ is in fact stable under conjugation by all
  of~$D$. We have therefore an induced morphism
  $$
  D\to \Aut(L),
  $$
  which induces an injective morphism $D/Z\to \Aut(L)$. Its image is
  contained in the neutral component of~$\Aut(L)$. Since~$L$ is
  semisimple, the latter is equal to the adjoint group of~$G$ (see,
  e.g.,~\cite[p.\,244, Prop.\,30, (ii)]{lie6}). It follows that the
  connected semisimple group~$G\subset \SU(V)$ has rank~$r-1$; it
  follows that $G=\SU(V)$ (e.g., by the Borel--de Siebenthal Theorem:
  the group $G$ coincides with the connected component of the identity
  of the centralizer in $\SU(V)$ of the center of~$G$, for instance
  by~\cite[p.\,36, prop.\,13]{lie9}, and the center is contained in the
  group of scalar matrices by Schur's Lemma, so its centralizer is
  $\SU(V)$).
\end{proof}

This proposition is applied to a conjugate of the \emph{finite}
subgroup~$A$ of elements of the form
$$
\mathrm{diag}(e(xv_1/p),\ldots, e(xv_{d-1}/p))
$$
where $(v_1,\ldots, v_{d-1})$ are the critical values of~$f$; indeed,
the local monodromy computation implies that such a subgroup is
contained in a maximal compact subgroup of the monodromy group.

\end{document}